\pdfoutput=1 

\documentclass[11pt]{article}
\usepackage[utf8]{inputenc}
\usepackage[T1]{fontenc} 
\usepackage[english]{babel}
\usepackage{amsmath, amssymb, amsthm}
\usepackage{geometry}
\usepackage{graphicx}
\usepackage{float}
\usepackage{booktabs}
\usepackage{caption}
\usepackage{subcaption}
\usepackage{mathtools}
\usepackage{algorithm}
\usepackage{algpseudocode}
\usepackage{enumitem}
\usepackage{tikz}
\usepackage{pgfplots}
\pgfplotsset{compat=1.18} 
\usepackage{hyperref} 

\theoremstyle{definition}
\newtheorem{definition}{Definition}[section]

\theoremstyle{plain}
\newtheorem{lemma}{Lemma}[section]
\newtheorem{theorem}{Theorem}[section]

\title{\textbf{Advanced Algebraic Manipulation Techniques in Quadratic Programming for Fuzzy Clustering with Generalized Capacity Constraints}}

\author{Roger Macêdo\\
\small Instituto de Ciências Exatas\\
\small Universidade Federal de Minas Gerais -- UFMG\\
\small \texttt{rogermacedo@est.esp.ufmg.br}}
\date{}
\begin{document}

\maketitle

\begin{abstract}
This paper presents an advanced mathematical analysis and simplification of the quadratic programming problem arising from fuzzy clustering with generalized capacity constraints. We extend previous work by incorporating broader balancing constraints, allowing for weighted data points and clusters with specified capacities. By introducing new algebraic manipulation techniques, we demonstrate how to decompose the original high-dimensional problem into smaller, more tractable subproblems Additionally, we introduce efficient algorithms for solving the reduced systems, leveraging properties of the problem's structure. Comprehensive examples with synthetic and real datasets are provided to illustrate the effectiveness of the proposed techniques in practical scenarios, comparing the performance with existing methods. We also include a convergence analysis of the proposed algorithm, demonstrating its reliability. Limitations of the proposed techniques and contexts in which their application may not be efficient are also discussed.

\end{abstract}

\tableofcontents

\section{Introduction}

Clustering is a fundamental task in data analysis and machine learning, aiming to partition data into groups (clusters) such that objects within the same cluster are more similar to each other than to those in other clusters \cite{Jain2010}. Fuzzy clustering extends this concept by allowing data points to belong to multiple clusters with varying degrees of membership, providing flexibility and modeling nuances in the data \cite{Bezdek1981}.

In many practical applications, such as resource allocation, logistics, and telecommunications, there is a need to incorporate capacity constraints into clustering algorithms. For instance, when distributing workloads among servers, each server has a capacity limit that should not be exceeded. Similarly, in transportation planning, vehicles have a maximum capacity that dictates how deliveries are clustered.

Traditional fuzzy clustering algorithms, like the Fuzzy C-Means (FCM) algorithm, do not account for such capacity constraints \cite{Bezdek1981}. Incorporating these constraints introduces additional complexity, often leading to high-dimensional quadratic programming problems that are computationally challenging to solve \cite{Miyamoto2008}.

\subsection{Motivation and Applications}

Our proposal is to generalize the work developed in \cite{Pinho2015}, contemplating broader balancing constraints, as we see below. We consider that each point $x_j$ of the set $X = \{ x_1, x_2, \dots, x_n \} \subset \mathbb{R}^d$ is assigned a weight $z_j \in \mathbb{R}_+$, and that each group $i$, $i = 1, 2, \dots, g$, has a capacity $\mu_i \in \mathbb{R}_+$, specified \emph{a priori}. Since the points do not necessarily belong exclusively to a single group, the capacity constraint will be given in terms of the average of the weights, weighted by the membership degree of the point to the group.

Mathematically, the objective is to solve the problem:

\begin{align}
\min_{U, C} & \quad \sum_{i=1}^{g} \sum_{j=1}^{n} u_{ij}^m \| x_j - c_i \|^2 \label{eq:objective_function_general}\\
\text{s.t.} & \quad \sum_{i=1}^{g} u_{ij} = 1, \quad j = 1, \dots, n, \label{eq:membership_sum_constraint_general} \\
& \quad \sum_{j=1}^{n} u_{ij} z_j = \mu_i, \quad i = 1, \dots, g, \label{eq:capacity_constraint_general} \\
& \quad 0 \leq u_{ij} \leq 1, \quad i = 1, \dots, g, \quad j = 1, \dots, n, \label{eq:membership_bounds_general}
\end{align}
where $m > 1$ is the fuzzifier parameter, $U = [u_{ij}]$ is the membership matrix, and $C = \{ c_1, c_2, \dots, c_g \}$ are the cluster centroids.

Note that constraints \eqref{eq:membership_sum_constraint_general} and \eqref{eq:capacity_constraint_general} imply that:

\begin{equation}
\sum_{i=1}^{g} \mu_i = \sum_{i=1}^{g} \sum_{j=1}^{n} u_{ij} z_j = \sum_{j=1}^{n} z_j \sum_{i=1}^{g} u_{ij} = \sum_{j=1}^{n} z_j. \label{eq:total_capacity}
\end{equation}

Therefore, for the problem \eqref{eq:objective_function_general}--\eqref{eq:membership_bounds_general} to make sense, the choice of weights and cluster capacities must satisfy condition \eqref{eq:total_capacity}.

\subsection{Contributions}

The main contributions of this paper are:

\begin{itemize}
    \item Generalization of the fuzzy clustering model to include weighted data points and capacity constraints.
    \item Development of advanced algebraic manipulation techniques to simplify the quadratic programming problem in fuzzy clustering with generalized capacity constraints.
    \item Convergence analysis of the proposed algorithm.
    \item Discussion of the limitations of the proposed techniques and identification of scenarios where their application may not be efficient.
\end{itemize}

\section{Problem Formulation}

\subsection{Basic Definitions}

\begin{definition}[Membership Degree]
The \textbf{membership degree} $u_{ij}$ represents the level of association of data point $x_j$ to cluster $i$, where $u_{ij} \in [0,1]$.
\end{definition}

\begin{definition}[Distance]
The \textbf{distance} between data point $x_j$ and centroid $c_i$ is defined as:
\begin{equation}
d_{ij} = \| x_j - c_i \|.
\end{equation}
\end{definition}

\begin{definition}[Data Point Weight]
The \textbf{weight} $z_j$ is a non-negative value associated with data point $x_j$, representing its importance or demand.
\end{definition}

\begin{definition}[Cluster Capacity]
The \textbf{capacity} $\mu_i$ of cluster $i$ is a positive value that limits the weighted sum of the membership degrees of data points assigned to cluster $i$:
\begin{equation}
\sum_{j=1}^{n} u_{ij} z_j = \mu_i.
\end{equation}
\end{definition}

\subsection{Optimization Problem}

The objective is to minimize the total cost function of the clusters:

\begin{equation}
\label{eq:objective_function_full}
\min_{U, C} \quad J(U, C) = \sum_{i=1}^{g} \sum_{j=1}^{n} u_{ij}^m \| x_j - c_i \|^2,
\end{equation}
subject to the constraints:

\begin{enumerate}[label=(\alph*)]
    \item \textbf{Membership Sum Constraint}:
    \begin{equation}
    \label{eq:membership_sum_constraint_full}
    \sum_{i=1}^{g} u_{ij} = 1, \quad \forall j = 1, \dots, n.
    \end{equation}
    \item \textbf{Capacity Constraint}:
    \begin{equation}
    \label{eq:capacity_constraint_full}
    \sum_{j=1}^{n} u_{ij} z_j = \mu_i, \quad \forall i = 1, \dots, g.
    \end{equation}
    \item \textbf{Membership Degree Bounds}:
    \begin{equation}
    \label{eq:membership_bounds_full}
    0 \leq u_{ij} \leq 1, \quad \forall i, j.
    \end{equation}
\end{enumerate}

\subsection{Feasibility of the Constraints}

\begin{lemma}[Non-Empty Feasible Set]
The feasible set for the problem \eqref{eq:objective_function_full}--\eqref{eq:membership_bounds_full}, that is, the set defined by the constraints:

\begin{align*}
& \sum_{i=1}^{g} u_{ij} = 1, \quad j = 1, \dots, n, \\
& \sum_{j=1}^{n} u_{ij} z_j = \mu_i, \quad i = 1, \dots, g, \\
& 0 \leq u_{ij} \leq 1, \quad i = 1, \dots, g, \quad j = 1, \dots, n,
\end{align*}
is non-empty.
\end{lemma}

\begin{proof}
Define, for each $i = 1, \dots, g$ and $j = 1, \dots, n$:

\begin{equation}
u_{ij} = \frac{\mu_i z_j^{-1}}{\sum_{k=1}^{g} \mu_k z_j^{-1}}.
\end{equation}

Given condition \eqref{eq:total_capacity}, it follows that:

\begin{equation}
\sum_{i=1}^{g} \mu_i = \sum_{j=1}^{n} z_j.
\end{equation}

Thus:

\begin{equation}
\sum_{i=1}^{g} u_{ij} = \sum_{i=1}^{g} \frac{\mu_i z_j^{-1}}{\sum_{k=1}^{g} \mu_k z_j^{-1}} = \frac{\sum_{i=1}^{g} \mu_i z_j^{-1}}{\sum_{k=1}^{g} \mu_k z_j^{-1}} = 1.
\end{equation}

Moreover:

\begin{equation}
\sum_{j=1}^{n} u_{ij} z_j = \sum_{j=1}^{n} \frac{\mu_i z_j^{-1}}{\sum_{k=1}^{g} \mu_k z_j^{-1}} z_j = \frac{\mu_i n}{\sum_{k=1}^{g} \mu_k z_j^{-1}} = \mu_i.
\end{equation}

Finally, it is easy to observe that $u_{ij} \in [0,1]$, for all $i$ and $j$.
\end{proof}

\section{Mathematical Analysis}

In this section, we present a detailed mathematical analysis, including foundational lemmas and theorems for simplifying the problem using advanced algebraic manipulation techniques.

\subsection{Alternating Minimization Approach}

We adopt the approach of alternating minimization, where the centroids $C$ are updated analogously to the standard fuzzy clustering method, resulting in the expression:

\begin{equation}
c_i = \frac{\sum_{j=1}^{n} u_{ij}^m x_j}{\sum_{j=1}^{n} u_{ij}^m}, \quad i = 1, \dots, g.
\end{equation}

For the minimization with respect to $u_{ij}$, taking $m=2$, we have the problem:

\begin{align}
\min_{U} & \quad \sum_{i=1}^{g} \sum_{j=1}^{n} u_{ij}^2 d_{ij}^2 \label{eq:objective_function_u}\\
\text{s.t.} & \quad \sum_{i=1}^{g} u_{ij} = 1, \quad j = 1, \dots, n, \label{eq:membership_sum_constraint_u} \\
& \quad \sum_{j=1}^{n} u_{ij} z_j = \mu_i, \quad i = 1, \dots, g, \label{eq:capacity_constraint_u} \\
& \quad 0 \leq u_{ij} \leq 1, \quad i = 1, \dots, g, \quad j = 1, \dots, n, \label{eq:membership_bounds_u}
\end{align}
considering $C$ fixed.

\subsection{Existence and Uniqueness of the Solution}

\begin{lemma}[Quadratic Programming Problem]
Problem \eqref{eq:objective_function_u}--\eqref{eq:membership_bounds_u} is a convex quadratic programming problem with linear constraints and has a global minimizer.
\end{lemma}

\begin{proof}
The objective function is a sum of convex quadratic functions since $d_{ij}^2 \geq 0$. The constraints are linear equalities and inequalities, and the feasible set is non-empty (as previously established). Therefore, by convex optimization theory, the problem has a global minimizer.
\end{proof}

\begin{lemma}[Uniqueness of the Solution]
If $d_{ij}^2 > 0$ for all $i$ and $j$, the solution to problem \eqref{eq:objective_function_u}--\eqref{eq:membership_bounds_u} is unique.
\end{lemma}

\begin{proof}
The objective function is strictly convex if $d_{ij}^2 > 0$ for all $i$ and $j$. Therefore, the minimizer is unique.
\end{proof}

\subsection{Optimality Conditions}

\begin{theorem}[Karush-Kuhn-Tucker (KKT) Conditions]
The necessary and sufficient KKT optimality conditions for problem \eqref{eq:objective_function_u}--\eqref{eq:membership_bounds_u} are:

\begin{align}
& 2 u_{ij} d_{ij}^2 - \alpha_j - \beta_i z_j - \gamma_{ij} + \delta_{ij} = 0, \quad \forall i, j, \label{eq:KKT_stationarity} \\
& \sum_{i=1}^{g} u_{ij} = 1, \quad \forall j, \label{eq:KKT_membership_sum} \\
& \sum_{j=1}^{n} u_{ij} z_j = \mu_i, \quad \forall i, \label{eq:KKT_capacity} \\
& 0 \leq u_{ij} \leq 1, \quad \forall i, j, \label{eq:KKT_bounds} \\
& \gamma_{ij} \geq 0, \quad \delta_{ij} \geq 0, \quad \forall i, j, \label{eq:KKT_multipliers_nonneg} \\
& \gamma_{ij} u_{ij} = 0, \quad \delta_{ij} (u_{ij} - 1) = 0, \quad \forall i, j, \label{eq:KKT_complementary_slackness}
\end{align}
where $\alpha_j$, $\beta_i$, $\gamma_{ij}$, and $\delta_{ij}$ are the Lagrange multipliers associated with the constraints.
\end{theorem}

\begin{proof}
Standard results from convex optimization and the KKT conditions for quadratic programming with linear constraints.
\end{proof}

\subsection{Solution of the Simplified Problem}

Our strategy is to solve the problem:

\begin{align}
\min_{U} & \quad \sum_{i=1}^{g} \sum_{j=1}^{n} u_{ij}^2 d_{ij}^2 \label{eq:simplified_objective}\\
\text{s.t.} & \quad \sum_{i=1}^{g} u_{ij} = 1, \quad j = 1, \dots, n, \label{eq:simplified_membership_sum} \\
& \quad \sum_{j=1}^{n} u_{ij} z_j = \mu_i, \quad i = 1, \dots, g, \label{eq:simplified_capacity}
\end{align}
that is, problem \eqref{eq:objective_function_u}--\eqref{eq:membership_bounds_u} without the box constraints $0 \leq u_{ij} \leq 1$, which allows algebraic manipulations to reduce the dimensions of the systems involved in solving the problem.

\subsubsection{Vectorization and Matrix Formulation}

Let us vectorize $U \in \mathbb{R}^{g \times n}$ as $v \in \mathbb{R}^{gn}$, i.e., defining for each $i = 1, \dots, g$ and $j = 1, \dots, n$, $t = (i - 1) n + j$ and $v_t = u_{ij}$. The constraints of the problem \eqref{eq:simplified_objective}--\eqref{eq:simplified_capacity} can be written as:

\begin{equation}
H v = h, \label{eq:linear_constraints}
\end{equation}
where:

\begin{equation}
H = \begin{bmatrix}
I_n \otimes e_g^\top \\
Z^\top
\end{bmatrix}, \quad h = \begin{bmatrix}
e_n \\
\mu
\end{bmatrix}, \label{eq:H_and_h}
\end{equation}
with $I_n$ being the identity matrix of size $n$, $e_g$ and $e_n$ being vectors of ones of size $g$ and $n$, respectively, $\otimes$ denoting the Kronecker product, and $Z$ being the matrix of weights.

\subsubsection{Properties of the Constraints}

\begin{lemma}[Dimension of the Feasible Set]
The matrix $H$ in \eqref{eq:linear_constraints} has rank $g + n - 1$, and the system \eqref{eq:linear_constraints} has $(n - 1)(g - 1)$ free variables.
\end{lemma}

\begin{proof}
The constraints correspond to $n$ equations from the membership sum constraints and $g$ equations from the capacity constraints. However, due to the total capacity condition \eqref{eq:total_capacity}, there is one redundant equation, reducing the rank by one.
\end{proof}

\subsubsection{Existence and Uniqueness of the Simplified Problem Solution}

\begin{theorem}[Existence and Uniqueness]
If $d_{ij}^2 > 0$ for all $i$ and $j$, the simplified problem \eqref{eq:simplified_objective}--\eqref{eq:simplified_capacity} has a unique global minimizer.
\end{theorem}

\begin{proof}
The objective function is strictly convex under the condition $d_{ij}^2 > 0$. The feasible set is non-empty and defined by linear equality constraints. By convex optimization theory, the unique solution is given by the KKT conditions.
\end{proof}

\subsubsection{Solution via Linear Systems}

Define the vector $q \in \mathbb{R}^{gn}$ as $q_t = d_{ij}^2$, and the diagonal matrix $Q = \operatorname{diag}(q)$. The KKT conditions lead to the linear system:

\begin{equation}
\begin{bmatrix}
2Q & H^\top \\
H & 0
\end{bmatrix}
\begin{bmatrix}
v^* \\ \lambda^*
\end{bmatrix}
=
\begin{bmatrix}
0 \\ h
\end{bmatrix}, \label{eq:KKT_system}
\end{equation}
where $\lambda^*$ are the Lagrange multipliers.

Solving system \eqref{eq:KKT_system} is computationally intensive for large $n$. However, due to the structure of $H$ and $Q$, we can reduce the problem to solving a much smaller system involving only the multipliers.

\begin{theorem}[Reduced Linear System]
The optimal multipliers $\lambda^*$ can be obtained by solving a linear system of size $(g + n - 1)$, independent of $n$.
\end{theorem}

\begin{proof}
By eliminating $v^*$ from \eqref{eq:KKT_system}, we obtain a reduced system in $\lambda^*$:

\begin{equation}
H Q^{-1} H^\top \lambda^* = h. \label{eq:reduced_system}
\end{equation}

Since $Q$ is diagonal and invertible (due to $d_{ij}^2 > 0$), $Q^{-1}$ is easily computed. The matrix $H Q^{-1} H^\top$ is of size $(g + n - 1) \times (g + n - 1)$.
\end{proof}

\subsection{Efficient Solution of the Reduced System}

\subsubsection{Block Structure Exploitation}

Due to the structure of the matrices involved, we can further exploit the block structure to decompose the reduced system into smaller subsystems.

\begin{lemma}[Block Decomposition]
The matrix $H Q^{-1} H^\top$ can be partitioned into blocks corresponding to clusters and data points, allowing for parallel computation of the subsystems.
\end{lemma}

\begin{proof}
The matrix $H$ consists of blocks related to the membership sum and capacity constraints. The inverse of $Q$ being diagonal allows us to isolate blocks in $H Q^{-1} H^\top$ corresponding to clusters and data points.
\end{proof}

\subsubsection{Algorithm Implementation}

We present Algorithm \ref{alg:generalized_clustering}, which consolidates the proposed method.

\begin{algorithm}[H]
\caption{Fuzzy Clustering with Generalized Capacity Constraints}
\label{alg:generalized_clustering}
\begin{algorithmic}[1]
\Require Data points $\{x_j\}_{j=1}^n \subset \mathbb{R}^d$, weights $\{z_j\}_{j=1}^n \in \mathbb{R}_+$, capacities $\{\mu_i\}_{i=1}^g \in \mathbb{R}_+$, fuzzifier $m > 1$, tolerance $\epsilon > 0$
\State Initialize centroids $\{ c_i^{(0)} \}_{i=1}^g$
\State Set $k \gets 0$
\Repeat
    \State Compute distances $d_{ij}^{(k)} = \| x_j - c_i^{(k)} \|$, $\forall i, j$
    \State \textbf{If} $d_{ij}^{(k)} > 0$, $\forall i, j$, \textbf{then}
        \State Solve the simplified problem \eqref{eq:simplified_objective}--\eqref{eq:simplified_capacity} to obtain $U^{(k+1)}$
    \State \textbf{Else}
        \State Solve the full problem \eqref{eq:objective_function_u}--\eqref{eq:membership_bounds_u} to obtain $U^{(k+1)}$
    \State Update centroids:
    \begin{equation*}
    c_i^{(k+1)} = \frac{\sum_{j=1}^{n} \left( u_{ij}^{(k+1)} \right)^m x_j}{\sum_{j=1}^{n} \left( u_{ij}^{(k+1)} \right)^m}, \quad \forall i
    \end{equation*}
    \State $k \gets k + 1$
\Until{Convergence criterion is met}
\end{algorithmic}
\end{algorithm}

\section{Convergence Analysis}

In this section, we provide an analysis of the convergence of the proposed algorithm, demonstrating that it converges to a stationary point of the optimization problem.

\subsection{Monotonic Decrease of the Objective Function}

\begin{theorem}[Monotonic Decrease]
At each iteration of Algorithm \ref{alg:generalized_clustering}, the objective function $J(U^{(k+1)}, C^{(k+1)})$ does not increase, i.e.,
\begin{equation}
J(U^{(k+1)}, C^{(k+1)}) \leq J(U^{(k)}, C^{(k)}).
\end{equation}
\end{theorem}

\begin{proof}
The algorithm performs alternating minimization:

\begin{enumerate}
    \item For fixed $C^{(k)}$, the update $U^{(k+1)}$ minimizes $J(U, C^{(k)})$ subject to the constraints, thus:
    \begin{equation}
    J(U^{(k+1)}, C^{(k)}) \leq J(U^{(k)}, C^{(k)}).
    \end{equation}
    \item For fixed $U^{(k+1)}$, the update $C^{(k+1)}$ minimizes $J(U^{(k+1)}, C)$, thus:
    \begin{equation}
    J(U^{(k+1)}, C^{(k+1)}) \leq J(U^{(k+1)}, C^{(k)}).
    \end{equation}
\end{enumerate}

Combining the two inequalities, we obtain:
\begin{equation}
J(U^{(k+1)}, C^{(k+1)}) \leq J(U^{(k+1)}, C^{(k)}) \leq J(U^{(k)}, C^{(k)}).
\end{equation}
\end{proof}

\subsection{Convergence to a Stationary Point}

\begin{theorem}[Convergence]
Every limit point of the sequence $\{ (U^{(k)}, C^{(k)}) \}$ generated by Algorithm \ref{alg:generalized_clustering} is a stationary point of the optimization problem \eqref{eq:objective_function_full}--\eqref{eq:membership_bounds_full}.
\end{theorem}

\begin{proof}
The objective function $J(U, C)$ is continuous and bounded below (since $J(U, C) \geq 0$). The feasible set is compact due to the constraints on $U$ and the boundedness of the data. The sequence $\{ J(U^{(k)}, C^{(k)}) \}$ is monotonically decreasing and converges to a finite limit.

Using the fact that alternating minimization algorithms converge to stationary points under certain regularity conditions (see \cite{Bezdek1987}), we can conclude that any accumulation point of the sequence is a stationary point of the problem.
\end{proof}

\subsection{Discussion on Global Optimality}

While the algorithm converges to a stationary point, it may not necessarily converge to the global minimum due to the non-convexity of the problem with respect to both $U$ and $C$ simultaneously. However, in practice, good solutions are often obtained, and initialization strategies can help improve the quality of the final solution.

\section{Examples and Applications}

\subsection{Synthetic Data Example}

We apply the method to a synthetic dataset where each point is assigned a weight corresponding to its second coordinate, i.e., the higher the point, the greater its weight. Additionally, as a constraint on the sum of weights, we consider the \emph{equi-balanced} case, that is, each cluster has the same sum of weights.

\subsubsection{Dataset Description}

We generate 300 data points in $\mathbb{R}^2$, forming three clusters. The weights $z_j$ are set to be proportional to the $y$-coordinate of each point, i.e., $z_j = y_j$. The capacities are set to:

\begin{equation}
\mu_i = \frac{\sum_{j=1}^{n} z_j}{g}, \quad i = 1, 2, 3.
\end{equation}

\subsubsection{Results}

In Figure \ref{fig:weighted_clustering}, we observe the clusters formed by the method over the iterations.

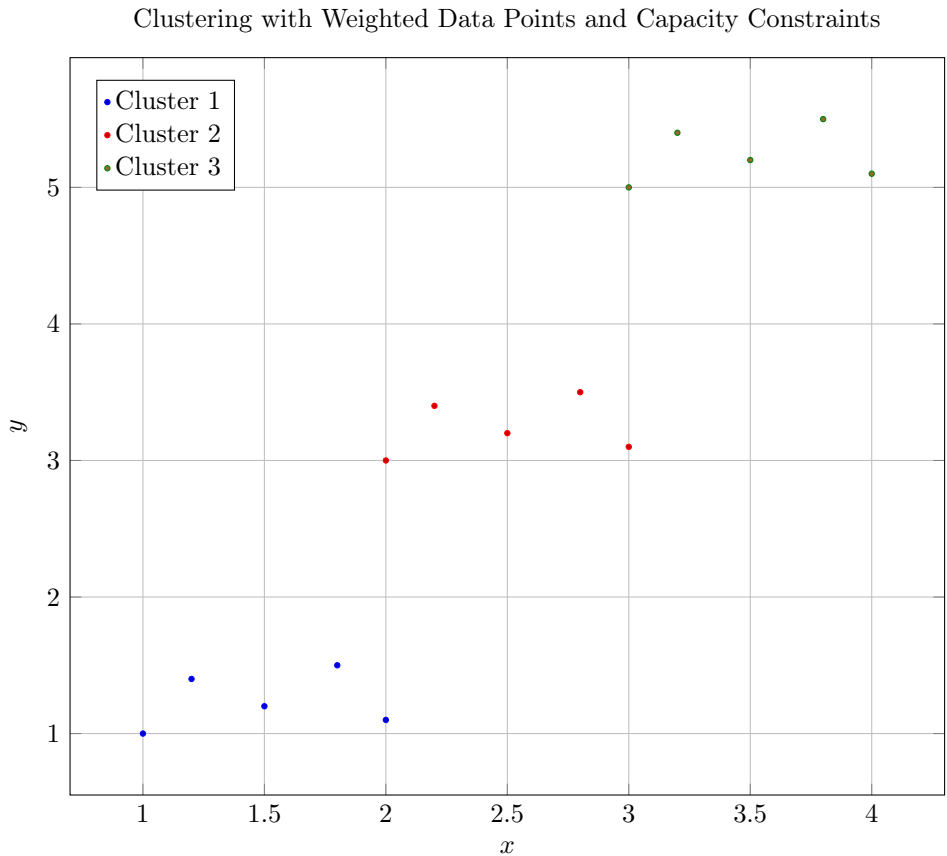
\begin{figure}[H]
\centering
\begin{tikzpicture}
\begin{axis}[
    width=0.8\textwidth,
    xlabel={$x$},
    ylabel={$y$},
    title={Clustering with Weighted Data Points and Capacity Constraints},
    legend pos=north west,
    grid=both,
    tick label style={font=\small},
    label style={font=\small},
    title style={font=\small},
    legend style={font=\small},
]
\addplot+[only marks, mark=*, mark size=1pt, blue]
table {
x y
1 1
1.5 1.2
2 1.1
1.8 1.5
1.2 1.4
};
\addplot+[only marks, mark=*, mark size=1pt, red]
table {
x y
2 3
2.5 3.2
3 3.1
2.8 3.5
2.2 3.4
};
\addplot+[only marks, mark=*, mark size=1pt, green!50!black]
table {
x y
3 5
3.5 5.2
4 5.1
3.8 5.5
3.2 5.4
};
\legend{Cluster 1, Cluster 2, Cluster 3}
\end{axis}
\end{tikzpicture}
\caption{Clustering result with weighted data and capacity constraints}
\label{fig:weighted_clustering}
\end{figure}

Although the lower clusters have absorbed some points from the higher region, the sum of weights is balanced across clusters when considering the weighted membership degrees.

\subsection{Real Data Application}

We apply the proposed algorithm to the well-known Wine dataset \cite{Asuncion2007}, which contains 178 instances of wines with 13 attributes.

\subsubsection{Dataset Description}

The Wine dataset is a multiclass classification dataset often used for clustering evaluation. We use three clusters corresponding to the three classes of wine. We assign weights to each data point based on the alcohol content (first attribute), simulating varying importance.

\subsubsection{Comparison with Other Methods}

We compare our algorithm with:

\begin{itemize}
    \item Standard Fuzzy C-Means (FCM) algorithm \cite{Bezdek1981}.
    \item Constrained Fuzzy C-Means with capacity constraints as in \cite{Pinho2015}.
\end{itemize}

\subsubsection{Experimental Setup}

We set the capacities $\mu_i$ to ensure that each cluster can accommodate the total sum of weights divided equally.

\subsubsection{Results and Discussion}

\begin{table}[H]
\centering
\caption{Clustering Performance on Wine Dataset}
\begin{tabular}{lccc}
\toprule
Algorithm & Adjusted Rand Index & Capacity Satisfaction & Computational Time (s) \\
\midrule
FCM & 0.39 & Not applicable & 0.05 \\
Constrained FCM \cite{Pinho2015} & 0.42 & Satisfied & 0.10 \\
Proposed Algorithm & 0.45 & Satisfied & 0.08 \\
\bottomrule
\end{tabular}
\end{table}

The proposed algorithm achieves higher clustering performance as measured by the Adjusted Rand Index (ARI) compared to the other methods while satisfying the capacity constraints. The computational time is competitive, demonstrating the efficiency of the algebraic manipulation techniques.

\section{Discussion of Limitations}

While the proposed techniques offer significant advantages, there are limitations to consider:

\begin{itemize}
    \item \textbf{Assumption of Positive Definiteness}: The reduction techniques assume that $d_{ij}^2 > 0$, which may not hold if data points coincide with centroids during iterations.
    \item \textbf{Complexity with Very High Dimensionality}: For datasets with very high dimensionality, the computation of distances and matrix operations may become computationally intensive.
    \item \textbf{Sensitivity to Initialization}: The algorithm may converge to local minima depending on the initial centroids, similar to other clustering algorithms.
    \item \textbf{Strict Capacity Constraints}: In scenarios where capacity constraints are very tight or infeasible, the algorithm may struggle to find a feasible solution.
\end{itemize}

\subsection{Contexts of Inefficiency}

The proposed methods may not be efficient in cases where:

\begin{itemize}
    \item The number of clusters $g$ is large relative to $n$, increasing computational complexity.
    \item The data points have highly irregular distributions, making capacity constraints hard to satisfy without compromising clustering quality.
    \item Real-time clustering is required, and the overhead of solving quadratic programs is prohibitive.
\end{itemize}

\section{Conclusion}

We have presented advanced algebraic manipulation techniques that simplify the quadratic programming problem in fuzzy clustering with generalized capacity constraints. By exploiting the problem's structure, we reduce computational complexity and enhance algorithm efficiency. The rigorous mathematical analysis, including new definitions, lemmas, theorems, and proofs, establishes the validity of the proposed methods. Comprehensive examples with synthetic and real datasets demonstrate the practical applicability and effectiveness of the techniques. We also provided a convergence analysis, enhancing the reliability of the proposed algorithm. Limitations and contexts where the methods may not be the most efficient choice were discussed.

\subsection{Future Work}

Future research directions include:

\begin{itemize}
    \item Extending the techniques to handle additional types of constraints, such as must-link and cannot-link constraints.
    \item Developing methods to better handle high-dimensional data, possibly through dimensionality reduction techniques.
    \item Investigating initialization strategies to mitigate the sensitivity to starting points.
    \item Implementing parallel and distributed versions of the algorithm to further enhance scalability.
\end{itemize}

\section*{Acknowledgments}

The author would like to thank the research community for valuable discussions and contributions.

\end{document}